\documentclass[oneside,english]{amsart}
\usepackage[T1]{fontenc}
\usepackage[latin9]{inputenc}
\usepackage[active]{srcltx}
\usepackage{amstext}
\usepackage{amsthm}
\usepackage{amssymb}

\makeatletter
\numberwithin{equation}{section}
\theoremstyle{plain}
\newtheorem{thm}{\protect\theoremname}[section]
  \theoremstyle{plain}
  \newtheorem{lem}[thm]{\protect\lemmaname}
  \theoremstyle{plain}
  \newtheorem{cor}[thm]{\protect\corollaryname}
  \theoremstyle{remark}
  \newtheorem{rem}[thm]{\protect\remarkname}

\usepackage{tikz-cd}

\newcommand{\mfg}{\mathfrak{g}}
\newcommand{\mfp}{\mathfrak{p}}

\newcommand{\mfo}{\mathfrak{o}}
\newcommand{\mfA}{\mathfrak{A}}
\newcommand{\mfP}{\mathfrak{P}}

\newcommand{\F}{\ensuremath{\mathbb{F}}}

\newcommand{\Z}{\ensuremath{\mathbb{Z}}}

\DeclareMathOperator{\Irr}{Irr}
\DeclareMathOperator{\Ind}{Ind}

\DeclareMathOperator{\Stab}{Stab}
\DeclareMathOperator{\tr}{tr}

\DeclareMathOperator{\Ker}{Ker}

\DeclareMathOperator{\End}{End}

\DeclareMathOperator{\rank}{rank}

\DeclareMathOperator{\GL}{GL}
\DeclareMathOperator{\M}{M}

\DeclareMathOperator{\chara}{char} 

\newcommand{\iso}{\mathbin{\kern.15em\widetilde{\hphantom{\hspace{.6em}}}\kern-.98em\rightarrow\kern.05em}}
\newcommand{\longiso}{\mathbin{\kern.3em\widetilde{\hphantom{\hspace{1.1em}}}\kern-1.55em\longrightarrow\kern.1em}}

\renewcommand{\theenumi}{{\arabic{enumi})}}
\renewcommand{\labelenumi}{\theenumi}

\date{}

\newcommand{\Amin}{\mfA_{\mathrm{m}}} 
\newcommand{\Pmin}{\mfP_{\mathrm{m}}}
\newcommand{\Umin}{U_{\mathrm{m}}}

\newcommand{\Amax}{\mfA_{\mathrm{M}}} 
\newcommand{\Pmax}{\mfP_{\mathrm{M}}}
\newcommand{\Umax}{U_{\mathrm{M}}}

\newcommand{\Hmin}{H_{\mathrm{m}}}
\newcommand{\Jmin}{J_{\mathrm{m}}}
\newcommand{\Hmax}{H_{\mathrm{M}}}
\newcommand{\Jmax}{J_{\mathrm{M}}}

\newcommand{\JmM}{J_{\mathrm{m,M}}}

\newcommand{\thetam}{\theta_{\mathrm{m}}}
\newcommand{\thetaM}{\theta_{\mathrm{M}}}
\newcommand{\etam}{\eta_{\mathrm{m}}}
\newcommand{\etaM}{\eta_{\mathrm{M}}}

\newcommand{\hatetaM}{\hat{\eta}_{\mathrm{M}}}

\newcommand{\emin}{e_{\mathrm{m}}}
\newcommand{\emax}{e_{\mathrm{M}}}

\newcommand{\betam}{\beta_{\mathrm{m}}}

\providecommand{\corollaryname}{Corollary}
  
  \providecommand{\lemmaname}{Lemma}
\providecommand{\theoremname}{Theorem}

\makeatother

\usepackage{babel}
  \providecommand{\corollaryname}{Corollary}
  \providecommand{\lemmaname}{Lemma}
  \providecommand{\remarkname}{Remark}
\providecommand{\theoremname}{Theorem}

\begin{document}

\title[Regular representations of $\GL_{N}$]{The regular representations of $\GL_{N}$ over finite local principal
ideal rings}

\author{Alexander Stasinski and Shaun Stevens}

\address{Alexander Stasinski, Department of Mathematical Sciences, Durham
University, South Rd, Durham, DH1 3LE, UK}

\email{alexander.stasinski@durham.ac.uk}

\address{Shaun Stevens, School of Mathematics, University of East Anglia,
Norwich Research Park, Norwich NR4 7TJ, UK}

\email{shaun.stevens@uea.ac.uk}
\begin{abstract}
Let $\mathfrak{o}$ be the ring of integers in a non-Archimedean local
field with finite residue field, $\mathfrak{p}$ its maximal ideal,
and $r\geq2$ an integer. An irreducible representation of the finite
group $G_{r}=\GL_{N}(\mathfrak{o}/\mathfrak{p}^{r})$ is called regular
if its restriction to the principal congruence kernel $K^{r-1}=1+\mfp^{r-1}\M_{N}(\mfo/\mfp^{r})$
consists of representations whose stabilisers modulo $K^{1}$ are
centralisers of regular elements in $\M_{N}(\mfo/\mfp)$. 

The regular representations form the largest class of representations
of $G_{r}$ which is currently amenable to explicit construction.
Their study, motivated by constructions of supercuspidal representations,
goes back to Shintani, but the general case remained open for a long
time. In this paper we give an explicit construction of all the regular
representations of~$G_{r}$.
\end{abstract}

\maketitle
 \renewcommand\labelenumi{\emph{(\roman{enumi})}} \renewcommand\theenumi\labelenumi 

\section{Introduction}

Let $F$ be a non-Archimedean local field with ring of integers $\mathfrak{o}$,
maximal ideal $\mathfrak{p}$ and finite residue field $\F_{q}$ of
characteristic $p$. The known explicit constructions of complex supercuspidal
representations of $\GL_{N}(F)$ are closely related to constructions
of representations of the maximal compact subgroup $\GL_{N}(\mfo)$.
These constructions go back to Shintani \cite{Shintani-68}, G\'e{}rardin
\cite{Gerardin}, Kutzko \cite{Kutzko-GL2-I,Kutzko-GL2-II}, Shalika
\cite{Shalika}, Howe \cite{Howe-Tame}, Carayol \cite{Carayol-supercusp},
culminating in the complete construction of supercuspidal representations
by Bushnell and Kutzko \cite{BushnellKutzko}. All of these constructions
are based on induction from compact mod centre subgroups of $\GL_{N}(F)$,
and as any compact subgroup is contained in a conjugate of $\GL_{N}(\mfo)$,
these constructions can also be seen as giving representations of
$\GL_{N}(\mfo)$. This connection goes further, because it has been
shown that every supercuspidal representation determines a unique
type on $\GL_{N}(\mfo)$, and two supercuspidals determine the same
type if and only if they differ by twisting by an unramified character;
see \cite[Appendix]{Henniart-appendix} and \cite{Vytas-unicity}. 

While the smooth representations of $\GL_{N}(F)$ have been studied
extensively, less is known about the representations of $\GL_{N}(\mfo)$.
The purpose of the current paper is to give a construction of a large
class of smooth complex representations of $\GL_{N}(\mfo)$ called
regular representations, which we now define. For any integer $r\geq1$
write $\mathfrak{o}_{r}$ for the finite local principal ideal ring
$\mathfrak{o}/\mathfrak{p}^{r}$. We will use $\mathfrak{p}$ to denote
also the maximal ideal in $\mathfrak{o}_{r}$. For any integer $r\geq2$,
let $G_{r}=\GL_{N}(\mathfrak{o}_{r})$. Every smooth, or equivalently,
continuous, representation of $\GL_{N}(\mfo)$ factors through some
group $G_{r}$. For any integer $i$ such that $r\geq i\geq1$, let
$G_{i}=\mbox{GL}_{N}(\mathfrak{o}_{i})$, let $\rho_{r,i}:G_{r}\rightarrow G_{i}$
be the homomorphism induced by the canonical map $\mathfrak{o}_{r}\rightarrow\mathfrak{o}_{i}$,
and let $K^{i}=\Ker\rho_{r,i}$ be the $i$-th principal congruence
kernel in $G_{r}$. Let $\mfg_{i}=\M_{N}(\mathfrak{o}_{i})$ denote
the algebra of $N\times N$ matrices over $\mathfrak{o}_{i}$. We
then have $K^{i}=1+\mathfrak{p}^{i}\mfg_{r}$. To any irreducible
representation $\pi$ of $G_{r}$ we can associate an adjoint orbit
(i.e., conjugation orbit, or similarity class) in $\mfg_{1}=\M_{N}(\F_{q})\cong K^{r-1}$,
via Clifford's theorem. The representation $\pi$ is called\emph{
regular} if the orbit consists of regular elements (i.e., the centraliser
in $\GL_{N}(\overline{\F}_{q})$ of any of its elements has minimal
dimension $N$). This is equivalent to the condition that the centraliser
in $G_{1}$ of any element in the orbit is abelian. The main reason
for focussing on regular representations is that their construction
lends itself well to the methods of Clifford theory. In particular,
the regular representations form the largest family of representations
which has so far been constructed explicitly for all $G_{r}$.

The study of regular representations of $G_{r}$ goes back to Shintani
\cite{Shintani-68}, and independently and later, Hill \cite{Hill_regular},
who constructed all the regular representations when $r$ is even.
The general case where $r$ is odd is much more difficult, requires
new ideas, and remained incomplete until the present paper. Assume
now that $r$ is odd. In \cite{Hill_semisimple_cuspidal} Hill constructed
all the cuspidal representations (i.e., the orbit has irreducible
characteristic polynomial), and in \cite{Hill_regular} he gave a
construction of so-called split regular representations (i.e., the
orbit has all its eigenvalues in $\F_{q}$). However, it was noted
by Takase \cite{Takase-16} that the results in \cite{Hill_regular}
do not give all the split regular representations, and in any case,
there exist many non-split non-cuspidal regular representations. While
the present work was in preparation, Krakovski, Onn and Singla \cite{KOS}
gave a construction of the regular representations of $G_{r}$ when
the residue characteristic $p$ is not $2$. In the present paper
we complete the picture by giving a construction of all the regular
representations of~$G_{r}$. For a somewhat more detailed comparative
account of the development of constructions of representations of
$G_{r}$, see \cite{reg-reps-survey}. 

In the present paper we give an explicit construction of all the regular
representations of $G_{r}$, without any assumption on the residue
characteristic of $\mfo$. Just like the other constructions mentioned
above, our approach is based on Clifford theory and orbits. A distinguishing
feature is that it is in some ways similar to the construction of
types on $\GL_{N}(\mfo)$ by Bushnell and Kutzko mentioned above.
The latter goes beyond regular representations but in a certain sense
only deals with semisimple elements, while we need to deal directly
with, for instance, regular nilpotent elements. Focussing on regular
elements has several technical advantages, but allowing non-semisimple
elements brings new phenomena, such as the non-triviality of the radical
of the form on $\Jmin^{1}/\Hmin^{1}$ (see Lemmas~\ref{lem:Radical}
and \ref{lem: dim-etam-etaM}). Since several of our lemmas hold also
for non-regular elements (but with more difficult proofs), it would
be interesting to know if the construction could be pushed further
to encompass both the regular representations and the supercuspidal
types on $\GL_{N}(\mfo)$ of Bushnell and Kutzko. 

Although there are many irreducible representations of $G_{r}$ which
are not regular, the regular representations are generic in the sense
that the regular elements in $\M_{N}(\overline{\F}_{q})$ are dense.
In particular, for $\GL_{2}(\mfo_{r})$ \emph{all }the irreducible
representations are either regular or pull-backs from $\GL_{2}(\mfo_{r-1})$.
Moreover, as noted by Lusztig, it is likely that the higher level
Deligne-Lusztig representations $\pm R_{T}^{G}(\theta)$, for $\theta$
regular and in general position, constructed in \cite{Lusztig-Fin-Rings}
and \cite{Alex_Unramified_reps} are regular in the case of $\mbox{GL}_{N}(\mathfrak{o}_{r})$,
and are therefore subsumed in the construction of the present paper.
That this is the case when $r$ is even is proved in \cite{Chen-Stasinski}.

\subsection*{Organisation of the paper}

In Section~\ref{sec:Parahoric} we define parahoric subalgebras and
associated filtrations of subgroups of $G_{r}$, using flags of $\mathfrak{o}_{r}$-modules
in $\mfg_{r}$. These are finite versions of the subgroups associated
to lattice chains in \cite[Section~1]{BushnellKutzko}. The particular
subalgebra $\Amin$, determined by the characteristic polynomial of
a regular element in $\mfg_{r}$ plays a central role in our construction.
In Section~\ref{sec:Chars}, we describe characters of certain abelian
groups defined earlier, and record well-known results about the existence
of ``Heisenberg lifts''. In Section~\ref{sec:Construction} we
give the construction of regular representations of $G_{r}$. The
reader who would like a quick overview of the steps of the construction,
illustrated by a diagram, may look at the discussion preceding Lemma~\ref{lem:psi-extn-theta}.
Our main theorem summarises the consequences of the construction for
the description of representations and appears Theorem~\ref{thm:Main-regular}.
Finally, Section~\ref{sec:Concluding-remarks} collects a few concluding
remarks.

\subsection*{\label{sec:Not-Pre}Notation and conventions}

If $G$ is a finite group, we will write $\Irr(G)$ for the set of
isomorphism classes of complex irreducible representations of $G$.
For convenience, we will always consider an element $\pi\in\Irr(G)$
as a representation, rather than an equivalence class of representations,
that is, we identify $\pi\in\Irr(G)$ with any representative in its
isomorphism class. One can view $\Irr(G)$ as the set of irreducible
characters of $G$, but we prefer to consider representations. If
$G$ is abelian, we will often refer to a one-dimensional representation
of $G$ as a character. If $H\subseteq G$ is a subgroup and $\pi$
is any representation of $G$ we write $\pi|_{H}$ for the restriction
of $\pi$ to $H$. If, moreover, $\sigma$ is a representation of
$H$, we will write $\Irr(G\mid\sigma)$ for the subset of $\Irr(G)$
consisting of representations which have $\sigma$ as an irreducible
constituent when restricted to $H$. 

Depending on the context, we use the notation $\M_{N}(\mfo_{r})$,
$\mfg_{r}$ or $E$, respectively, to denote the algebra of $N\times N$
matrices over $\mfo_{r}$. The notation $E$ appears (only) in Section~\ref{sec:Parahoric}
where the algebra is seen as the endomorphisms of $\mfo_{r}^{N}$,
and $\mfg_{r}$ appears in the rest of the paper, where it plays the
role of the $\mfo_{r}$-points of the Lie algebra of $\GL_{N}$.

We use $\varpi$ to denote a fixed choice of generator of $\mfp\subset\mfo_{r}$.

We will make free use of some well known results from Clifford theory
(see \cite[Section~2]{Alex_smooth_reps_GL2}). 

\subsubsection*{Acknowledgement. }

The first author was supported by EPSRC grant EP/K024779/1. The second
author was supported by EPSRC grant EP/H00534X/1. We wish to thank
Uri Onn for alerting us to several typos in a previous version of
this paper.

\section{\label{sec:Parahoric}parahoric subalgeberas and filtrations of subgroups}

The main goal of this section is to attach to any element in $\M_{N}(\F_{q})$
a parahoric subgroup of $G_{r}$ together with a natural filtration.
These filtrations are finite versions of the ones defined by lattice
chains in \cite[Section~1]{BushnellKutzko}.

Let $V$ be a free $\mathfrak{o}_{r}$-module of rank $N$, and let
$\overline{V}=V\otimes_{\mathfrak{o}_{r}}\F_{q}\cong V/\mathfrak{p}V$
(an $N$-dimensional vector space over $\F_{q}$). Let $\rho_{r,1}:V\rightarrow\overline{V}$
denote the canonical map. Let $E=\End_{\mathfrak{o}_{r}}(V)$ and
$\overline{E}=\End_{k}(\overline{V})\cong E\otimes_{\mathfrak{o}_{r}}\F_{q}$.
Let 
\[
V=V_{0}\supset V_{1}\supset\dots\supset V_{e}=0
\]
be a flag of $\mathfrak{o}_{r}$-modules with $e\geq1$ an integer,
and such that $V_{i}$ is free for $0\leq i\leq e-1$. Let $\overline{V}=\overline{V}_{0}\supset\overline{V}_{1}\supset\dots\supset\overline{V}_{e}=0$
be the flag of $\F_{q}$-vector spaces obtained by setting $\overline{V}_{i}=V_{i}\otimes_{\mathfrak{o}_{r}}\F_{q}\cong V_{i}/\mathfrak{p}V_{i}$.
We have $\rank_{\mathfrak{o}_{r}}V_{i}=\dim_{\F_{q}}\overline{V}_{i}$
for $0\leq i\leq e-1$. 
\begin{lem}
Let notation be as above. We then have a chain of $\mathfrak{o}_{r}$-modules
\[
V=L_{0}\supset L_{1}\supset\dots\supset L_{e}\supset\dots\supset L_{er}=0,
\]
where $L_{i+ej}=\mathfrak{p}^{j}\rho_{r,1}^{-1}(\overline{V}_{i})=\mathfrak{p}^{j}(V_{i}+\mathfrak{p}V)$,
for $0\leq i\leq e-1$ and $0\leq j$. 
\end{lem}
\begin{proof}
The only thing that requires proof is that all the inclusions are
strict. If $0\leq i\leq e-1$ and $0\leq j\leq r-1$ then multiplication
by $\mathfrak{p}^{j}$ induces an isomorphism $L_{i}/L_{i+1}\cong L_{i+ej}/L_{i+ej+1}$,
while the map $\rho_{r,1}$ induces an isomorphism $L_{i}/L_{i+1}\cong\overline{V}_{i}/\overline{V}_{i+1}$;
in particular, $L_{i+ej}/L_{i+ej+1}$ is non-zero. 
\end{proof}
We put $N_{i}=\rank_{\mathfrak{o}_{r}}V_{i}$, for $i=0,\ldots,e$,
so that $N=N_{0}>N_{1}>\cdots>N_{e}=0$. Since $\mathfrak{o}_{r}$
is a self-injective ring (cf.~\cite[3.12]{Lam-LectModRings}), a
free submodule of a free $\mathfrak{o}_{r}$-module is a direct summand.
Hence a basis for a free submodule of $V$ can always be extended
to a basis for $V$, and so there exists a basis $\{x_{1},\dots,x_{N}\}$
for $V$ such that $\{x_{1},\dots,x_{N_{i}}\}$ is a basis for $V_{i}$,
for $i=0,\ldots,e$. Then the image $\{x_{1}+\mathfrak{p}V,\dots,x_{N}+\mathfrak{p}V\}$
of this basis under the map $V\rightarrow\overline{V}$ is a basis
for $\overline{V}$ such that $\{x_{1}+\mathfrak{p}V,\dots,x_{N_{i}}+\mathfrak{p}V\}$
is a basis for $\overline{V}_{i}$, and the $\mathfrak{o}_{r}$-module
$L_{i+ej}$ has basis consisting of (the non-zero elements in) 
\begin{equation}
\{\varpi^{j}x_{1},\ldots,\varpi^{j}x_{N_{i}},\varpi^{j+1}x_{N_{i}+1},\ldots,\varpi^{j+1}x_{N}\},\label{eq:latticebasis}
\end{equation}
for $0\leq i\leq e-1$ and $0\leq j$.

We define the $\mfo_{r}$-algebras 
\begin{align*}
P & =\{x\in E\mid xV_{i}\subseteq V_{i}\text{ for all }0\leq i\leq e\},\\
\overline{P} & =\{x\in\overline{E}\mid x\overline{V_{i}}\subseteq\overline{V_{i}}\text{ for all }0\leq i\leq e\}.
\end{align*}
Algebras of this form are called \emph{parabolic subalgebras} of $E$
and $\overline{E}$, respectively. Similarly, we define the algebra
\[
\mathfrak{A}=\{x\in E\mid xL_{i}\subseteq L_{i}\text{ for all }0\leq i\leq er\},
\]
and an algebra of this form is called a \emph{parahoric subalgebra}
of $E$. The algebra $P$ has a (two-sided) ideal $I$ given by 
\[
I=\{x\in E\mid xV_{i}\subseteq V_{i+1}\text{ for all }0\leq i\leq e-1\},
\]
and the analogous ideal $\overline{I}$ in $\overline{P}$ is defined
in the obvious way. Similarly, the algebra $\mathfrak{A}$ has an
ideal $\mathfrak{P}$ given by 
\[
\mathfrak{P}=\{x\in E\mid xL_{i}\subseteq L_{i+1}\text{ for all }0\leq i\leq er-1\}.
\]
We have $I^{e}=\overline{I}\vphantom{I}^{e}=\mathfrak{P}^{er}=0$,
so the ideals are nilpotent. We remark that, since we have the periodicity
relation $L_{i+ej}=\mathfrak{p}^{j}L_{i}$, we also have 
\begin{align*}
\mathfrak{A} & =\{x\in E\mid xL_{i}\subseteq L_{i}\text{ for all }0\leq i\leq e-1\},\\
\mathfrak{P} & =\{x\in E\mid xL_{i}\subseteq L_{i+1}\text{ for all }0\leq i\leq e-1\}.
\end{align*}
Using the basis $\{x_{1},\ldots,x_{N}\}$ above to identify $V$ with
$\mathfrak{o}_{r}^{N}$, whence $E$ with $\M_{N}(\mathfrak{o}_{r})$,
these algebras and ideals have convenient matrix pictures. For example,
writing the matrix of $a\in E$ with respect to this basis as $(a_{jk})$,
we see that $a\in P$ if and only if $a_{jk}=0$ whenever there is
an integer $0\leq i\leq e-1$ such that $j>N_{i}\geq k$. Thus $P$
is the algebra of block upper-triangular matrices, with block sizes
$N_{1}':=N_{e-1}-N_{e},\ldots,N_{e}':=N_{0}-N_{1}$, while $I$ is
its ideal of strictly block upper-triangular matrices. Similarly,
the fact that~\eqref{eq:latticebasis} gives a basis for $L_{i}$
implies that $\mathfrak{A}$ is the algebra of block matrices which
are block upper-triangular modulo $\mathfrak{p}$, with the same block
sizes, while $\mathfrak{P}$ is its ideal of matrices which are strictly
block upper-triangular modulo~$\mathfrak{p}$. Writing $\M_{m\times n}(R)$
for the set of $m\times n$ matrices over a ring $R$ (not necessarily
with identity), we have 
$$
\includegraphics{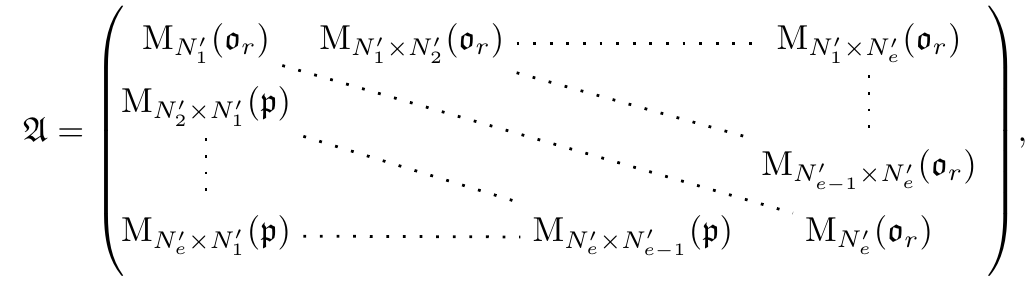}
$$
$$
\includegraphics{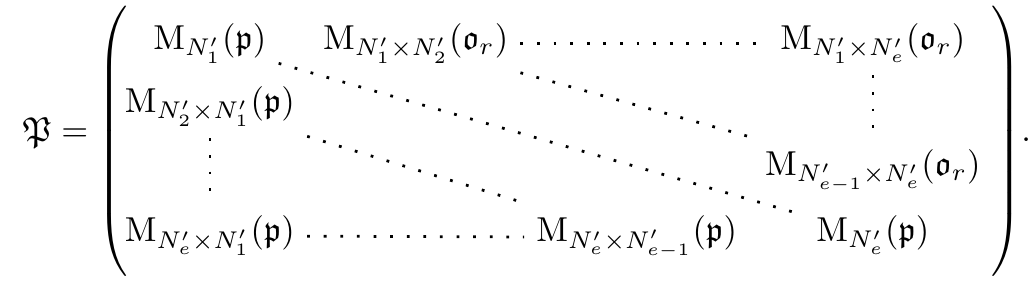}
$$
Since, with respect to the basis $\{x_{1}+\mathfrak{p}V,\dots,x_{N}+\mathfrak{p}V\}$
of $\overline{V}$, the parabolic algebra $\overline{P}$ also consists
of block upper-triangular matrices with the same block-sizes, the
map $\rho_{r,1}:E\rightarrow\overline{E}$ induces surjections $P\rightarrow\overline{P}$,
$I\rightarrow\overline{I}$, and we get:
\begin{lem}
\label{lem:P and I}~ 
\begin{enumerate}
\item \label{enu:P}$\mathfrak{A}=\rho_{r,1}^{-1}(\overline{P})=P+\mathfrak{p}E$,
\item \label{enu:I}$\mathfrak{P}=\rho_{r,1}^{-1}(\overline{I})=I+\mathfrak{p}E$. 
\end{enumerate}
\end{lem}
We remark that Lemma~\ref{lem:P and I} implies that $\mathfrak{A}/\mathfrak{P}\cong\overline{P}/\overline{I}$,
which is semisimple, so that $\mathfrak{P}$ is the Jacobson radical
of $\mathfrak{A}$.

Although we always have a surjection $\rho_{r,1}:I^{m}\rightarrow\overline{I}\vphantom{I}^{m}$,
in general $\mathfrak{P}^{m}$ is not equal to $I^{m}+\mathfrak{p}E$,
for $m\geq2$. However, we can use the matrix description of $\mathfrak{P}$
above to obtain a similar description of $\mathfrak{P}^{m}$, by multiplying
elementary matrices with respect to the basis $\{x_{1},\ldots,x_{N}\}$.
Indeed, a straightforward induction shows that, for $m\geq0$, the
ideal $\mathfrak{P}^{m}$ consists of block matrices whose $(i,j)$-block
has entries in $\mathfrak{p}^{\lceil(m+i-j)/e\rceil}$, where $\lceil y\rceil$
denotes the least integer greater than or equal to $y$. 
\begin{lem}
\label{lem:shift} For $m\geq0$ and $0\leq k\leq e(r-1)+1-m$, we
have: 
\begin{enumerate}
\item \label{enu:shift1}$\mathfrak{P}^{m}L_{i}=L_{i+m}$, for any $i\geq0$,
\item \label{enu:shift2}$\mathfrak{P}^{m}=\{x\in E\mid xL_{i}\subseteq L_{i+m}\text{ for all }k\leq i\leq k+e-1\}$,
\item \label{enu:shift3}$\mathfrak{P}^{m}=\{x\in E\mid x\mathfrak{P}^{k}\subseteq\mathfrak{P}^{k+m}\}$. 
\end{enumerate}
\end{lem}
\begin{proof}
Given the description of $\mathfrak{P}^{m}$ above, it is straightforward
to check that the image of the basis~\eqref{eq:latticebasis} of~$L_{i}$
under elementary matrices in $\mathfrak{P}^{m}$ contains the basis~\eqref{eq:latticebasis}
of~$L_{i+m}$, and \ref{enu:shift1} follows. Similarly, it is straightforward
to check that the matrix description of 
\[
\{x\in E\mid xL_{i}\subseteq L_{i+m}\text{ for all }0\leq i\leq e-1\}
\]
is the same as that for $\mathfrak{P}^{m}$ above. Now \ref{enu:shift2}
follows since $xL_{i}\subseteq L_{i+m}$ if and only if $xL_{i+e}\subseteq L_{i+e+m}$,
whenever $0\leq i\leq e(r-1)-m$. Finally, for \ref{enu:shift3},
suppose $x\in E$ is such that $x\mathfrak{P}^{k}\subseteq\mathfrak{P}^{k+m}$
so that $x\mathfrak{P}^{k}L_{i}\subseteq L_{i+k+m}$, for $i=0,\ldots,e-1$.
But then \ref{enu:shift1} implies $xL_{i+k}\subseteq L_{i+k+m}$,
for $i=0,\ldots,e-1$, and \ref{enu:shift2} implies $x\in\mathfrak{P}^{m}$. 
\end{proof}
As an immediate corollary, we get:
\begin{cor}
\label{cor:AP-Formulae}~
\begin{enumerate}
\item \label{enu:Formula 1}$\mathfrak{p}\mathfrak{A}=\mathfrak{Ap}=\mathfrak{P}^{e}$, 
\item \label{enu:Formula 2}$\mathfrak{P}^{m}=\mathfrak{P}^{m+1}$ if and
only if $m\geq er$. 
\end{enumerate}
\end{cor}
Note also that we have $\mathfrak{P}^{e(r-1)}=\mathfrak{p}^{r-1}P$
and $\mathfrak{P}^{e(r-1)+1}=\mathfrak{p}^{r-1}I$. Let $\tr:E\rightarrow\mathfrak{o}_{r}$
denote the trace map. 
\begin{lem}
\label{lem:Trace-lemma}Let $x\in E$, and let $m$ be an integer
such that $0\leq m\leq e(r-1)+1$. Then $\tr(\mathfrak{P}^{m}x)=\{0\}$
if and only if $x\in\mathfrak{P}^{e(r-1)+1-m}$. 
\end{lem}
\begin{proof}
Note that one implication is clear, since $\mathfrak{P}^{e(r-1)+1}\subseteq I$
so $\tr(\mathfrak{P}^{e(r-1)+1})=0$.

We first prove the opposite implication for $m=0$, so we assume that
$x\in E$ is such that $\tr(\mathfrak{A}x)=\{0\}$. It is easy to
show (e.g.~using elementary matrix considerations) that the trace
form $E\times E\rightarrow\mathfrak{o}_{r}$ given by $(\alpha,\beta)\mapsto\tr(\alpha\beta)$
is non-degenerate. Similarly, it is also easy to show that for $\gamma\in E$,
the condition $\tr(P\gamma)=\{0\}$ implies that $\gamma\in I$. Hence,
since $\mathfrak{p}E\subset\mathfrak{A}$, the condition $\tr(\mathfrak{A}x)=\{0\}$
implies that $x\in\mathfrak{p}^{r-1}E$. Furthermore, since $P\subset\mathfrak{A}$,
the condition $\tr(\mathfrak{A}x)=\{0\}$ implies that $x\in I\cap\mathfrak{p}^{r-1}E=\mathfrak{p}^{r-1}I=\mathfrak{P}^{e(r-1)+1}$,
as required.

Now suppose $x\in E$ is such that $\tr(\mathfrak{P}^{m}x)=\{0\}$.
Then $\tr(\mathfrak{A}(\mathfrak{P}^{m}x))=\{0\}$ so the case $m=0$
implies that $\mathfrak{P}^{m}x\subseteq\mathfrak{P}^{e(r-1)+1}$.
Now Lemma~\ref{lem:shift},~\ref{enu:shift3} implies that $x\in\mathfrak{P}^{e(r-1)+1-m}$,
as required. 
\end{proof}
Define the groups 
\[
U=U^{0}=\mathfrak{A}^{\times},\quad U^{m}=1+\mathfrak{P}^{m},\text{ for }m\geq1.
\]
The group $U$ is called a \emph{parahoric subgroup }of $E^{\times}$.
We have a filtration 
\[
U\supset U^{1}\supset\dots\supset U^{er-1}\supset U^{er}=\{1\},
\]
where the inclusions are strict thanks to Lemma~\ref{cor:AP-Formulae},
\ref{enu:Formula 2}. We also define $U^{i}=\{1\}$ for $i>er$.

Since $\mathfrak{P}$ is a (two-sided) ideal in $\mathfrak{A}$, each
group $U^{m}$ is normal in $U$. Moreover, if $1+x\in U^{m}$, and
$1+y\in U^{n}$, then 
\[
(1+x)(1+y)\equiv1+x+y\equiv(1+y)(1+x)\pmod{\mathfrak{P}^{m+n}},
\]
so we have the commutator relation 
\[
[U^{m},U^{n}]\subseteq U^{m+n}.
\]
Thus in particular, the group $U^{m}$ is abelian whenever $m\geq er/2$,
that is, when $m\geq\lceil\frac{er}{2}\rceil$.

For every $m\geq1$ we have an isomorphism 
\[
U^{m}/U^{m+1}\longiso\mathfrak{P}^{m}/\mathfrak{P}^{m+1},\qquad(1+x)U^{m+1}\longmapsto x+\mathfrak{P}^{m+1},
\]
and since $\mathfrak{p}\mathfrak{P}^{m}\subseteq\mathfrak{P}^{m+e}\subseteq\mathfrak{P}^{m+1}$,
we have an action of $\mathfrak{o}_{r}/\mathfrak{p}\cong\F_{q}$ on
$\mathfrak{P}^{m}/\mathfrak{P}^{m+1}$, for each $m\geq0$. This makes
$\mathfrak{P}^{m}/\mathfrak{P}^{m+1}$ a finite dimensional vector
space over the finite field $\F_{q}$, where the action of $\F_{q}$
is compatible with the group structure. Hence $\mathfrak{P}^{m}/\mathfrak{P}^{m+1}$
is an elementary abelian group.

By choosing a basis, we may identify $V$ with $\mathfrak{o}_{r}^{N}$,
$\overline{V}$ with $\F_{q}^{N}$, $E$ with $\M_{N}(\mathfrak{o}_{r})$,
$\overline{E}$ with $\M_{N}(\F_{q})$, and $E^{\times}$ with $G_{r}$.
These identifications will remain in force throughout the rest of
this paper. From now on, let $\Omega_{r}\subset\mfg_{r}$ be an orbit
under the adjoint (conjugation) action of $G_{r}$. Write $\Omega_{1}$
for the image of $\Omega_{r}$ in $\mfg_{1}$. We will associate a
certain parahoric subalgebra to $\Omega_{1}$, which will be denoted
by $\Amin$. Let 
\[
\prod_{i=1}^{h}f_{i}(x)^{m_{i}}\in\F_{q}[x]
\]
be the characteristic polynomial of $\Omega_{1}$ (i.e., the characteristic
polynomial of any element in $\Omega_{1}$), where the $f_{i}(x)$
are distinct and irreducible of degree $d_{i}$, for $i=1,\dots,h$.
This determines a partition of $N$:
\[
\lambda=(d_{1}^{m_{1}},\dots,d_{h}^{m_{h}})=(\underbrace{d_{1},d_{1},\dots,d_{1}}_{m_{1}\text{ times}},\dots,\underbrace{d_{h},d_{h},\dots,d_{h}}_{m_{h}\text{ times}}).
\]
We define $\Amin\subseteq\mfg_{r}$ to be the preimage of the standard
parabolic subalgebra of $\mfg_{1}$ corresponding to $\lambda$ (i.e.,
the block upper-triangular subalgebra whose block sizes are given
by $\lambda$, in the order given above). Moreover, we let $\Amax=\mfg_{r}=\M_{N}(\mfo_{r})$
be the full matrix algebra. Let $\Pmin$ and $\Pmax$ be the corresponding
Jacobson radicals of  $\Amin$ and $\Amax$, respectively. Then $\Pmin$
is the pre-image under $\rho_{r,1}:\mfg_{r}\rightarrow\mfg_{1}$ of
the strict block-upper subalgebra of $\Amin$, and similarly $\Pmax$
is the pre-image of $0$, that is, $\Pmax=\mfp\mfg_{r}$. For $*\in\{\mathrm{m},\mathrm{M}\}$
we have the corresponding groups 
\[
U_{*}=U_{*}^{0}=\mfA_{*}^{\times},\quad U_{*}^{i}=1+\mfP_{*}^{i},\qquad\text{for }i\geq1,
\]
and the filtrations 
\[
U_{*}\supset U_{*}^{1}\supset\dots\supset U_{*}^{e_{*}r}=\{1\},
\]
where $e_{*}=e(\mfA_{*})$. Note that $\Umax^{i}=K^{i}$ and $\emax=1$,
while $\emin=m_{1}+\dots+m_{h}$ is the number of blocks in $\Amin$.

By definition, we have $\Amax\supseteq\Amin$, and the label $\mathrm{m}$
here stands for ``minimal'', while $\mathrm{M}$ stands for ``maximal''.
From the definitions, we have 
\begin{align*}
\Umin/\Umin^{1} & \cong\prod_{i=1}^{h}\GL_{d_{i}}(\F_{q}))^{m_{i}},\quad & \Amin/\Pmin & \cong\prod_{i=1}^{h}\M_{d_{i}}(\F_{q})^{m_{i}},\\
\Umax/\Umax^{1} & =G_{r}/K^{1}\cong G_{1}, & \Amax/\Pmax & \cong\mfg_{1}.
\end{align*}
Note that if $\Omega_{1}$ has irreducible characteristic polynomial,
then $\Amin=\Amax=\mfg_{r}$.

Given an element $\beta\in\mfg_{r}$ we will denote its image in $\mfg_{1}$
by $\bar{\beta}$. Similarly, if $\beta\in\Omega_{r}\cap\Amin$, we
will let $\betam$ denote the image of $\beta$ in $\Amin/\Pmin$.
Note that by, for example, the rational Jordan normal form, $\Omega_{r}\cap\Amin$
is non-empty. Up to $G_{r}$-conjugation we can also arrange the diagonal
irreducible blocks of $\bar{\beta}$, and hence of $\betam$ in any
order. In particular, we can find a $\beta\in\Omega_{r}\cap\Amin$
such that
\begin{equation}
\betam=\underbrace{\betam^{1}\oplus\dots\oplus\betam^{1}}_{m_{1}\text{ times}}\oplus\dots\oplus\underbrace{\betam^{h}\oplus\dots\oplus\betam^{h}}_{m_{h}\text{ times}},\label{eq:betam}
\end{equation}
where each $\betam^{i}\in\M_{d_{i}}(\F_{q})$ has irreducible characteristic
polynomial $f_{i}(x)$. 

There are several equivalent characterisations of regular elements
in $\mfg_{1}$. One of the simplest is that an element in $\mfg_{1}$
is regular if its centraliser in $\mfg_{1}$ has dimension $N$ (as
$\F_{q}$-vector space).\emph{ }One can also define regular elements
in $\mfg_{i}$ for $r\geq i>1$, as those elements whose centraliser
in $\mfg_{i}$ has $\mfo_{i}$-rank $N$. A result of Hill \cite[Theorem~3.6]{Hill_regular}
implies that an element in $\mfg_{i}$ is regular if and only if its
image in $\mfg_{1}$ is regular.
\begin{lem}
\label{lem:Centr-in-Am}Let $\beta\in\Omega_{r}\cap\Amin$. If $\beta$
is regular, then $C_{G_{r}}(\beta)\subseteq\Amin^{\times}$.
\end{lem}
\begin{proof}
If $\beta$ is regular, we have $C_{G_{r}}(\beta)=\mfo_{r}[\beta]^{\times}$,
so $\beta\in\Amin$ implies that $C_{G_{r}}(\beta)\subset\Amin$,
since $\Amin$ is an algebra.
\end{proof}
\begin{lem}
\label{lem:Centr_A/P-orders}We have
\[
|C_{\Amin/\Pmin}(\betam)|=q^{N}=|C_{\mfg_{1}}(\bar{\beta})|.
\]
\end{lem}
\begin{proof}
The isomorphism $\Amin/\Pmin\cong\prod_{i=1}^{h}\M_{d_{i}}(\F_{q})^{m_{i}},$
induces an isomorphism 
\[
C_{\Amin/\Pmin}(\betam)\cong\prod_{i=1}^{h}C_{\M_{d_{i}}(\F_{q})}(\betam^{i})^{m_{i}},
\]
so $|C_{\Amin/\Pmin}(\betam)|=\prod_{i=1}^{h}q^{d_{i}m_{i}}=q^{N}$.
The second equality follows by definition of regularity of $\bar{\beta}$.
\end{proof}
Set $l=\lceil\frac{r}{2}\rceil$, $l'=\lfloor\frac{r}{2}\rfloor$,
so that $l+l'=r$. The relations $\Amax\supseteq\Amin\supseteq\Pmin\supseteq\Pmax$
imply that for every $i\geq1$, $\Pmax^{i}=\mfp^{i}\mfg_{r}$ is a
two-sided ideal in $\Amin$. For $\beta\in\Omega_{r}\cap\mfA$ and
$*\in\{\mathrm{m},\mathrm{M}\}$, we can therefore define the following
groups
\begin{align*}
C & =C_{G_{r}}(\beta),\\
J_{*}^{1} & =(C\cap U_{*}^{1})U_{*}^{e_{*}l'},\\
H_{*}^{1} & =(C\cap U_{*}^{1})U_{*}^{e_{*}l'+1}.
\end{align*}
Note that $\Jmax^{1}=(C\cap K^{1})K^{l'}$ and $\Hmax^{1}=(C\cap K^{1})K^{l}$,
and that both of these groups are normalised by $CK^{l'}$, since
$C$ normalises both $K^{1}$ and $K^{l'}$, and $[K^{l'},K^{1}]\subseteq K^{l}\subseteq K^{l'}$.
Moreover, we define the group
\[
\JmM=(C\cap\Umin^{1})K^{l'}.
\]
We have the following diagram of subgroups, where the vertical and
slanted lines denote inclusions (we have only indicated the inclusions
which are relevant to us and repeat the definitions of the groups,
for the reader's convenience).
$$
\includegraphics{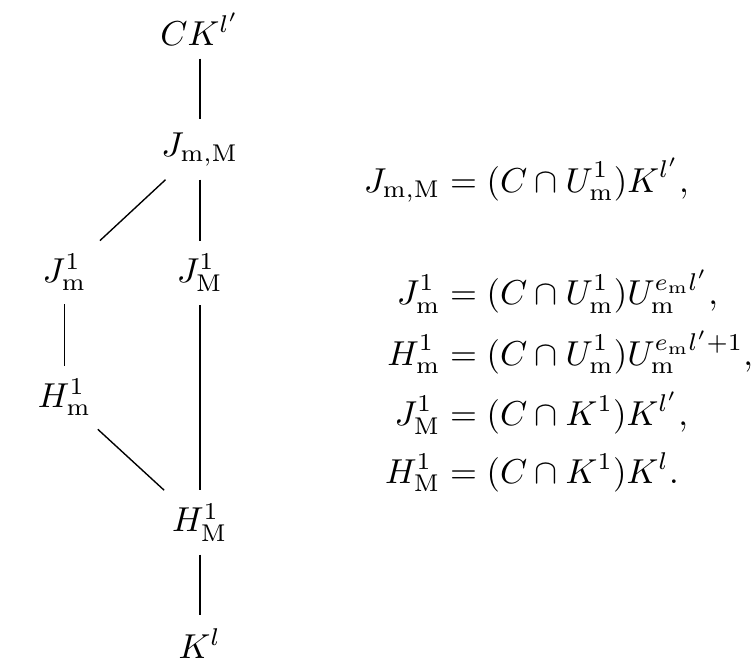}
$$
We explain the non-trivial inclusions in the above diagram. We have
$\Pmin\supseteq\Pmax$, and so $\Umin^{1}\supseteq K^{1}$. By Corollary~\ref{cor:AP-Formulae},
we get
\[
\Umin^{e_{\mathrm{m}}l'+1}=1+\mfp^{l'}\Pmin\supseteq1+\mfp^{l'}\Pmax=K^{l};
\]
thus $\Hmin^{1}\supseteq\Hmax^{1}$. Moreover, 
\[
\Umin^{e_{\mathrm{m}}l'}=1+\mfp^{l'}\Amin\subseteq1+\mfp^{l'}\Amax=K^{l'},
\]
so $\JmM$ contains both $\Jmin^{1}$ and $\Jmax^{1}$ as subgroups. 

The following lemma will be a crucial step in the construction of
representations, and is the main reason why we work with the algebra
$\Amin$ and its associated subgroups. 
\begin{lem}
\label{lem:p-Sylow}Suppose that $\Omega_{r}$ consists of regular
elements. Then there exists a $\beta\in\Omega_{r}$ such that $\JmM$
is a Sylow $p$-subgroup of $CK^{l'}$. 
\end{lem}
\begin{proof}
By the rational Jordan normal form, there is a $\beta\in\Omega_{r}\cap\Amin$.
Moreover, any permutation of the diagonal blocks can be achieved by
$G_{r}$-conjugation, so there is a $\beta\in\Omega_{r}\cap\Amin$
such that $\betam\in\Amin/\Pmin$ satisfies \eqref{eq:betam}. Assume
now that $\beta$ is chosen in this way. We have 
\begin{align*}
[CK^{l'}:\JmM] & =\frac{|CK^{l'}|/|K^{l'}|}{|\JmM|/|K^{l'}|}=\frac{|C/(C\cap K^{l'})|}{|(C\cap\Umin^{1})/(C\cap K^{l'})|}=\frac{|C|}{|C\cap\Umin^{1}|}.
\end{align*}
Thus we need to show that $C\cap\Umin^{1}$ is a Sylow $p$-subgroup
in $C$. Since $\beta$ is regular, $C$ is abelian, and by Lemma~\ref{lem:Centr-in-Am}
we have $C\subseteq\Umin=\Amin^{\times}$, so we have
\[
\frac{C}{C\cap\Umin^{1}}=\frac{C\cap\Umin}{C\cap\Umin^{1}}.
\]
Then the isomorphism $\Umin/\Umin^{1}\cong\prod_{i=1}^{h}\GL_{d_{i}}(\F_{q})^{m_{i}}$
induces an isomorphism
\[
\frac{C\cap\Umin}{C\cap\Umin^{1}}\cong\prod_{i=1}^{h}C_{\GL_{d_{i}}(\F_{q})}(\betam^{i})^{m_{i}}.
\]
Each $\betam^{i}$ has irreducible characteristic polynomial over
$\F_{q}$, so $\F_{q}[\betam^{i}]/\F_{q}$ is a field extension of
degree $d_{i}$. Since $C_{\GL_{d_{i}}(\F_{q})}(\betam^{i})=\F_{q}[\betam^{i}]^{\times}$,
we conclude that $p$ does not divide the order of $C_{\GL_{d_{i}}(\F_{q})}(\betam^{i})$.
Therefore, $p$ does not divide the order of $\frac{C}{C\cap\Umin^{1}}$,
so $\JmM$ is a Sylow $p$-subgroup of $CK^{l'}$. 
\end{proof}
We remark that the above lemma holds without the hypothesis that $\Omega_{r}$
consists of regular elements, but the proof is slightly easier in
the case of a regular orbit. Note also that when $\beta$ is as in
the above lemma, $\JmM$ is in fact normal in $CK^{l'}$, since $C=C\cap\Umax$
normalises $\JmM$ when $C$ is abelian. Thus $\JmM$ is the unique
Sylow $p$-subgroup of $CK^{l'}$. We will not need this fact.

\section{\label{sec:Chars}Characters and Heisenberg lifts}

As in the introduction, $F$ denotes the fraction field of $\mfo$.
Fix an additive character $\psi:F\rightarrow\mathbb{C}^{\times}$
which is trivial on $\mfo$ but not on $\mfp^{-1}$ (i.e., $\psi$
has conductor $\mfo$). For each $r\geq1$ we can view $\psi$ as
a character of the group $F/\mfp^{r}$ whose kernel contains $\mfo_{r}$.
We will use $\psi$ and the trace form $(x,y)\mapsto\tr(xy)$ on $\mfg_{r}$
to set up a duality between the groups $\Irr(K^{i})$ and $\mfg_{r-i}$,
for $i\geq r/2$. For $\beta\in\M_{N}(\mathfrak{o}_{r})$, define
a homomorphism $\psi_{\beta}:K^{i}\rightarrow\mathbb{C}^{\times}$
by
\begin{equation}
\psi_{\beta}(1+x)=\psi(\varpi^{-r}\tr(\hat{\beta}\hat{x})),\label{eq:character-psi-b}
\end{equation}
where $x\in\mfp^{i}\mfg_{r}$, and $\hat{\beta},\hat{x}\in\M_{N}(\mfo)$
denote arbitrary lifts of $\beta$ and $x$, respectively. The value
$\psi(\varpi^{-r}\tr(\hat{\beta}\hat{x}))$ is independent of the
choice of lifts, since $\psi$ is trivial on $\mfo$. For this reason,
we will abuse notation slightly from now on and write $\psi(\varpi^{-r}\tr(\beta x))$
instead of $\psi(\varpi^{-r}\tr(\hat{\beta}\hat{x}))$. The map $\beta\mapsto\psi_{\beta}$
is a homomorphism whose kernel is $\mfp^{r-i}\mfg_{r}$, thanks to
the non-degeneracy of the trace form. Hence it induces an isomorphism
\[
\mfg_{r}/\mfp^{r-i}\mfg_{r}\longiso\Irr(K^{i}),
\]
where we will usually identify $\mfg_{r}/\mfp^{r-i}\mfg_{r}$ with
$\mfg_{r-i}$. For $g\in G_{r}$ we have 
\[
\psi_{g\beta g^{-1}}(1+x)=\psi(\varpi^{-r}\tr(g\beta g^{-1}x))=\psi(\varpi^{-r}\tr(\beta g^{-1}xg))=\psi_{\beta}(1+g^{-1}xg).
\]
Let $\mathfrak{A},\mathfrak{P}$, and $U^{m},m\geq0$ be the objects
associated to an arbitrary flag of length $e$, as in Section~\ref{sec:Parahoric}.
Let $n$ and $m$ be two integers such that $e(r-1)+1\geq n>m\geq n/2>0$.
Then $U^{m}/U^{n}$ is abelian, and we have an isomorphism 
\[
\mathfrak{P}^{m}/\mathfrak{P}^{n}\longiso U^{m}/U^{n},\qquad x+\mathfrak{P}^{n}\longmapsto(1+x)U^{n}.
\]
Each $a\in\mfg_{r}$ defines a character $\mfg_{r}\rightarrow\mathbb{C}^{\times}$
via $x\mapsto\psi(\varpi^{-r}\tr(ax))$, and this defines an isomorphism
$\mfg_{r}\rightarrow\Irr(\mfg_{r})$. For any subgroup $S$ of $\mfg_{r}$,
define 
\[
S^{\perp}=\{x\in\mfg_{r}\mid\psi(\varpi^{-r}\tr(xS))=1\}.
\]
Using the isomorphism $\mfg_{r}\rightarrow\Irr(\mfg_{r})$, we can
identify $S^{\perp}$ with the group of characters of $\mfg_{r}$
which are trivial on $S$. 

We generalise the definition of $\psi_{\beta}$ to allow $\beta$
to lie in an appropriate power of $\mfP$. For any $\beta\in\mathfrak{P}^{e(r-1)+1-n}$
define a character $\psi_{\beta}:U^{m}\rightarrow\mathbb{C}^{\times}$
by 
\[
\psi_{\beta}(1+x)=\psi(\varpi^{-r}\tr(\beta x)).
\]

\begin{lem}
\label{lem:characters}Let $e(r-1)+1\geq n>m\geq n/2>0$. Then
\begin{enumerate}
\item \label{enu:ortho}For any integer $i$ such that $0\leq i\leq e(r-1)+1$,
we have 
\[
(\mathfrak{P}^{i})^{\perp}=\mathfrak{P}^{e(r-1)+1-i}.
\]
\item \label{enu:chariso}The map $\beta\mapsto\psi_{\beta}$ induces an
isomorphism 
\[
\mathfrak{P}^{e(r-1)+1-n}/\mathfrak{P}^{e(r-1)+1-m}\longiso\Irr(U^{m}/U^{n}).
\]
\end{enumerate}
\end{lem}
\begin{proof}
Let $\rho_{r}:\mfo\rightarrow\mfo_{r}$ be the canonical map. For
any $x\in\mfg_{r}$ the set $\varpi^{-r}\rho_{r}^{-1}(\tr(x\mathfrak{P}^{i}))$
is a fractional ideal of $\mfo$, so by our choice of $\psi$ we have
\[
\psi(\varpi^{-r}\tr(x\mathfrak{P}^{i})):=\psi(\varpi^{-r}\rho_{r}^{-1}(\tr(x\mathfrak{P}^{i}))=1
\]
 if and only if $\varpi^{-r}\rho_{r}^{-1}(\tr(x\mathfrak{P}^{i})\subseteq\mfo$.
Thus $x\in(\mathfrak{P}^{i})^{\perp}$ if and only if $\tr(x\mathfrak{P}^{i})=0$
in $\mfo_{r}$, so \ref{enu:ortho} follows from Lemma~\ref{lem:Trace-lemma}.
Moreover, \ref{enu:ortho} together with the isomorphism $\mathfrak{P}^{m}/\mathfrak{P}^{n}\cong U^{m}/U^{n}$
implies \ref{enu:chariso}. 
\end{proof}
Let $G$ be a finite group and $N$ a normal subgroup, such that $G/N$
is an elementary abelian $p$-group. Then the group $G/N$ has a structure
of $\F_{p}$-vector space. Let $\chi$ be a one-dimensional representation
of $N$ which is stabilised by $G$. Define an alternating bilinear
form 
\[
h_{\chi}:G/N\times G/N\longrightarrow\mathbb{C}^{\times},\qquad h_{\chi}(xN,yN)=\chi([x,y])=\chi(xyx^{-1}y^{-1}).
\]
By bilinearity we simply mean that $h_{\chi}(xN,yzN)=h_{\chi}(xN,yN)h_{\chi}(xN,zN)$
for all $x,y,z\in G$, and similarly for linearity in the first variable.
This follows from the commutator relation $[x,yz]=[x,y]^{y}[x,z]$
and its analogue for the first variable. Note that linearity with
respect to scalar multiplication follows from this since for $\bar{n}\in\F_{p}$
we have $\bar{n}(xN)=x^{n}N$, for any lift $n\in\Z$ of $\bar{n}$.).
An easy computation shows that $h_{\chi}$ is well-defined, thanks
to the stability of $\chi$ under $G$. Define the subspace
\[
\overline{R}_{\chi}=\{xN\in G/N\mid h_{\chi}(xN,yN)=1\:\text{for all }y\in G\}.
\]
This is the \emph{radical} of the form $h_{\chi}$, and we say that
$h_{\chi}$ is \emph{non-degenerate} if $\overline{R}_{\chi}=0$.
We will make use of the following result (cf.~\cite[8.3.3]{Bushnell_Frohlich}): 
\begin{lem}
\label{lem:Bushnell-Frohlich}Assume that the form $h_{\chi}$ is
non-degenerate. Then there exists a unique $\eta_{\chi}\in\Irr(G\mid\chi)$,
and $\dim\eta_{\chi}=[G:N]^{1/2}$.
\end{lem}
Note that if $\chi$ and $\chi'$ are two representations of $N$
such that $\eta_{\chi}=\eta_{\chi'}$, then by Clifford's theorem,
the restriction $\eta_{\chi}|_{N}=\eta_{\chi'}|_{N}$ is a multiple
of $\chi$ and of $\chi'$, so we must have $\chi=\chi'$.

We will encounter situations where the form $h_{\chi}$ is not non-degenerate.
In these cases we will apply the following generalisation, which is
a corollary of the above lemma:
\begin{cor}
\label{cor:Bushnell-Frohlich}Let $G$, $N$ and $\overline{R}_{\chi}$
be as above, and let $R_{\chi}$ be the inverse image of $\overline{R}_{\chi}$
under the map $G\rightarrow G/N$. Then $\chi$ has an extension to
$R_{\chi}$, and for any extension $\tilde{\chi}$ of $\chi$, there
exists a unique $\eta_{\tilde{\chi}}\in\Irr(G\mid\tilde{\chi})$.
Moreover, 
\[
\dim\eta_{\tilde{\chi}}=[G:R_{\chi}]^{1/2}.
\]
\end{cor}
\begin{proof}
By definition, $\chi([R_{\chi},R_{\chi}])=1$, so $R_{\chi}/\Ker\chi$
is abelian and thus $\chi$ extends to $R_{\chi}$. Let $\tilde{\chi}$
be an extension and $x\in G$. Then, for any $r\in R_{\chi}$, we
have $[x,r]\in N$, so
\[
\tilde{\chi}([x,r])=\chi([x,r])=1.
\]
Hence $\tilde{\chi}(xrx^{-1})=\tilde{\chi}(r)$, that is, $x$ stabilises
$\tilde{\chi}$. Moreover, $R_{\chi}$ is normal in $G$ since for
any $x,y\in G$, $r\in R_{\chi}$, we have 
\begin{align*}
\chi([xrx^{-1},y]) & =\tilde{\chi}([xrx^{-1},y])=\tilde{\chi}(xrx^{-1}yxr^{-1}x^{-1}y^{-1})\\
 & =\tilde{\chi}(x^{-1}y^{-1}xrx^{-1}yxr^{-1})=\tilde{\chi}(x^{-1}y^{-1}xrx^{-1}yx)\tilde{\chi}(r^{-1})\\
 & =\tilde{\chi}(r)\tilde{\chi}(r^{-1})=1.
\end{align*}
We thus have a well-defined form $h_{\tilde{\chi}}$ on $G/R_{\chi}$,
which is non-degenerate. The remaining statements now follow immediately
from Lemma~\ref{lem:Bushnell-Frohlich}.
\end{proof}
In the situation of the above corollary, we call $\eta_{\tilde{\chi}}$
a \emph{Heisenberg lift} of $\chi$.

We will apply the above results to the group $J_{*}^{1}$ and its
normal subgroup $H_{*}^{1}$, for $*\in\{\mathrm{m},\mathrm{M}\}$.
We have an isomorphism 
\[
J_{*}^{1}/H_{*}^{1}\cong\frac{U_{*}^{e_{*}l'}}{(C\cap U_{*}^{e_{*}l'})U_{*}^{e_{*}l'+1}},
\]
and this is a quotient of $U_{*}^{e_{*}l'}/U_{*}^{e_{*}l'+1}\cong\overline{\mfA}_{*}/\overline{\mfP}_{*}$
(where $\overline{\mfA}_{*}$ is the image of $\mfA_{*}$ in $\mfg_{1}$,
as in Section~\ref{sec:Parahoric}), with the latter isomorphism
being induced by the map $1+\pi^{l'}x\mapsto\bar{x}+\mfP_{*}$. Since
$\overline{\mfA}_{*}/\overline{\mfP}_{*}$ is a product of additive
groups of matrix rings over $\F_{q}$, it is an elementary abelian
$p$-group (where as before $p=\chara\F_{q}$), and in fact has a
structure of $\F_{q}$-vector space. Thus $J_{*}^{1}/H_{*}^{1}$,
being a quotient of an elementary abelian group, is itself elementary
abelian.

\section{\label{sec:Construction}Construction of regular representations}

For any $x\in\mfg_{r}$, let $x_{i}$ denote the image of $x$ in
$\mfg_{i}$, for $r\geq i\geq1$. We also write $\bar{x}$ for $x_{1}$.
Recall from the paragraph preceding Lemma~\ref{lem:Centr-in-Am}
that a regular element in $\mfg_{i}$ can be defined by the property
of having centraliser in $\mfg_{i}$ of $\mfo_{i}$-rank $N$. A result
of Hill \cite[Corollary~3.7]{Hill_regular} implies that if $\beta_{i}\in\mfg_{i}$
is regular, then $C_{G_{i}}(\beta_{i})=\mfo_{i}[\beta_{i}]^{\times}$.
It follows that if $\beta$ is regular, then $C=C_{G_{r}}(\beta)$
is abelian and the homomorphisms
\[
\rho_{r,i}:C\rightarrow C_{G_{i}}(\beta_{i})
\]
are surjective for every $r\geq i\geq1$. 

Suppose that $\pi$ is an irreducible representation of $G_{r}$.
By Clifford's theorem, the restriction of $\pi$ to the abelian group
$K^{l}$ defines an orbit of characters $\psi_{\beta}$, and hence
by the results of Section~\ref{sec:Chars}, an orbit of elements
in $\mfg_{r}/\mathfrak{p}^{r-l}\mfg_{r}\cong\mfg_{l'}$ under the
conjugacy action of $G_{r}$ (i.e., the adjoint action). We call $\pi$
\emph{regular} if this orbit in $\mfg_{l'}$ consists of regular elements.
\\
\\
Fix an orbit $\Omega_{l'}\subset\mfg_{l'}$ consisting of regular
elements. We will construct all the irreducible representations of
$G_{r}$ with orbit $\Omega_{l'}$. When $r$ is even, the construction
is well known and amounts to taking any $\beta+\mfp^{l'}\mfg_{r}\in\Omega_{l'}$,
extending $\psi_{\beta}$ to $CK^{l'}$ and inducing to $G_{r}$.
To show that $\psi_{\beta}$ extends to $CK^{l'}$ is straightforward
in this case; see for example \cite[Theorem~4.1]{Hill_regular}.

From now on, assume that $r\geq3$ is odd, so that $l'=l-1$. We highlight
the hypotheses that will remain in force throughout this section:
\begin{enumerate}
\item $r\geq3$ is odd, 
\item $\beta\in\mfg_{r}$ is regular,
\item $\beta\in\Amin$ and the image $\betam\in\Amin/\Pmin\cong\prod_{i=1}^{h}\M_{d_{i}}(\F_{q})^{m_{i}}$
is
\[
\betam=\underbrace{\betam^{1}\oplus\dots\oplus\betam^{1}}_{m_{1}\text{ times}}\oplus\dots\oplus\underbrace{\betam^{h}\oplus\dots\oplus\betam^{h}}_{m_{h}\text{ times}},
\]
where $\betam^{i}\in\M_{d_{i}}(\F_{q})$ have irreducible characteristic
polynomial. 
\end{enumerate}
Before presenting the details of the construction, we give an informal
overview. Schematically, the construction is illustrated by the following
diagrams (dotted lines are extensions, dashed are Heisenberg lifts,
and the solid one between $\etam$ and $\eta$ is an induction):
$$
\includegraphics{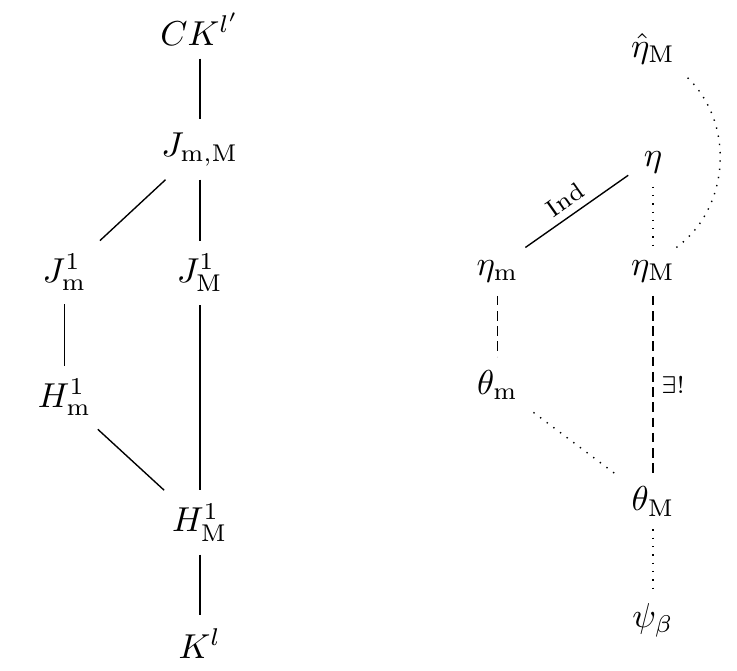}
$$
The diagram of representations on the right should be read from bottom
to top. We have seen in Lemma~\ref{lem:p-Sylow} that there exists
a $\beta\in\Omega_{r}$ such that $\JmM$ is a Sylow $p$-subgroup
of $CK^{l'}$, and fix one such $\beta$. We then show that the character
$\psi_{\beta}$ of $K^{l}$ has an extension $\thetaM$ to $\Hmax^{1}$
and that $\thetaM$ extends further to $\thetam$ on $\Hmin^{1}$.
Next, we use Corollary~\eqref{cor:Bushnell-Frohlich} to show that
there exists a unique representation $\etaM\in\Irr(\Jmax^{1}\mid\thetaM)$,
as well as a (non-unique) representation $\etam\in\Irr(\Jmin^{1}\mid\thetam)$.
Moreover, we compute the dimensions of $\etaM$ and $\etam$. Then
we show that $\eta:=\Ind_{\Jmin^{1}}^{\JmM}\etam$ has the same dimension
as $\etaM$, from which it follows that $\eta$ is an extension of
$\etaM$. We can then apply a general lemma to conclude that $\etaM$
has an extension $\hatetaM$ to $CK^{l'}$. Finally, we show that
by choosing all possible extensions $\thetaM$ of $\psi_{\beta}$
and all possible extensions $\hatetaM$ of $\etaM$, we have constructed
all the representations in $\Irr(CK^{l'}\mid\psi_{\beta})$, without
redundancy. Since $CK^{l'}=\Stab_{G_{r}}(\psi_{\beta})$, a standard
result from Clifford theory then yields all the irreducible representations
of $G_{r}$ with orbit $\Omega_{r}$ by induction from $CK^{l'}$. 

We now give the details and proofs of the construction. Some of the
steps can be carried out for the groups arising from the algebras
$\Amin$ and $\Amax$ simultaneously. For this purpose, we will let
$\mfA$ denote either $\Amin$ or $\Amax$, and let $\mfP$ be the
radical in $\mfA$, with $e=e(\mfA)$. The associated subgroups will
be denoted by $U^{i}$, $H^{1}$, $J^{1}$.
\begin{lem}
\label{lem:psi-extn-theta}The character $\psi_{\beta}$ has an extension
$\thetaM$ to $\Hmax^{1}$. Moreover, $\thetaM$ has an extension
$\thetam$ to $\Hmin^{1}$.
\end{lem}
\begin{proof}
By Lemma~\ref{lem:characters}\,\ref{enu:chariso}, if we take
\[
m=el'+1,\qquad n=2m-1=e(r-1)+1,
\]
then the coset $\beta+\mathfrak{\mfP}^{e_{}l'}$ defines a character
on $U^{el'+1}$, trivial on $U^{e(r-1)+1}$ by the same formula as
the one defining $\psi_{\beta}$. Since $\mathfrak{\mfP}_{}^{e_{}l'}=\mfp^{l'}\mfA_{}$,
we have a map 
\[
\mfA_{}/\mathfrak{\mfP}_{}^{e_{}l'}\longrightarrow\mfg_{r}/\mfp^{l'}\mfg_{r},
\]
which sends $\beta+\mathfrak{\mfP}_{}^{e_{}l'}$ to $\beta+\mfp^{l'}\mfg_{r}$
(note that this is neither surjective nor injective). Thus the different
choices of lift of the latter coset give the different choices of
extension of $\psi_{\beta}$ to $U_{}^{e_{}l'+1}$. Our element $\beta\in\mfA$
therefore gives rise to an extension (which we still denote by $\psi_{\beta})$
of $\psi_{\beta}$ to $U^{el'+1}$, defined by 
\[
\psi_{\beta}(1+x)=\psi(\varpi^{-r}\tr(\beta x)),\quad\text{for }x\in\mfP_{}^{e_{\mathrm{}}l'+1}.
\]
We now show the existence of the extensions $\thetaM$ and $\thetam$.
If $c\in C\cap U_{}^{1}$ and $x\in\mathfrak{\mfP}_{}^{e_{}l'+1}$,
then 
\begin{align*}
[c,1+x] & \in c(1+x)c^{-1}(1-x+\mfP_{}^{e_{}(r-1)+2})\\
 & =1+cxc^{-1}-x+\mfP_{}^{e_{}(r-1)+2}.
\end{align*}
Since we have 
\begin{equation}
U^{e_{}(r-1)+1}\subseteq\Ker\psi_{\beta},\label{eq:Umin-in-ker}
\end{equation}
we obtain 
\[
\psi_{\beta}([c,1+x])=\psi(\varpi^{-r}\tr(\beta(cxc^{-1}-x)))=\psi(\varpi^{-r}\tr(c\beta xc^{-1}-\beta x))=1,
\]
where we have used that $c$ commutes with $\beta$.

Thus $C\cap U_{}^{1}$ stabilises the character $\psi_{\beta}$ on
$U^{el'+1}$, and since $C\cap U^{1}$ is abelian, this implies that
$\psi_{\beta}$ extends to $H^{1}=(C\cap U^{1})U^{e_{}l'+1}$. 
\end{proof}
\begin{rem}
\label{rem:theta_0psi_beta}The extension $\thetaM$ can be written
as a character $\theta_{0}\psi_{\beta}$, where $\theta_{0}\in\Irr(C\cap K^{1})$
is a character which agrees with $\psi_{\beta}$ on $C\cap K^{l}$
and
\[
\theta_{0}\psi_{\beta}(zk):=\theta_{0}(z)\psi_{\beta}(k),
\]
for $z\in C\cap K^{1}$ and $k\in K^{1}$. We will use this later
in the proof of Lemma~\ref{lem:Stab-etaM}. One can write $\thetam$
similarly, but we will not need that.
\end{rem}
We fix arbitrary extensions $\thetaM$ and $\thetam$ as in the above
lemma. For $*\in\{\mathrm{m},\mathrm{M}\}$, we will now construct
the irreducible representations $\eta_{*}$ of $J_{*}^{1}$ containing
$\theta_{*}$. In particular, we will show that there exists a unique
representation $\etaM$ of $\Jmax^{1}$ containing $\thetaM$. We
will treat both cases simultaneously, denoting either $\thetam$ or
$\thetaM$ by $\theta$. We need to verify the hypotheses of Corollary~\ref{cor:Bushnell-Frohlich}.
To this end, first note that $\theta_{}$ is stabilised by $J_{}^{1}$.
Indeed, it is enough to show that $U_{}^{e_{}l'}$ stabilises $\theta_{}$.
For $x\in\mfP_{}^{e_{}l'}$, $c\in(C\cap U_{}^{1})$ and $y\in\mfP_{}^{e_{}l'+1}$,
we have
\begin{align*}
[1+x,c(1+y)] & \in(1+x)c(1+y)(1-x+x^{2}+\mfP^{e(r-1)+1})(1-y+\mfP^{e_{}(r-1)+1})c^{-1}\\
 & \subseteq(c+xc-cx+cy)(1-y)c^{-1}+\mfP^{e(r-1)+1}\\
 & \subseteq1+x-cxc^{-1}+\mfP^{e_{}(r-1)+1}.
\end{align*}
Hence, since $\psi_{\beta}$ is trivial on $U^{e(r-1)+1}$ (see \eqref{eq:Umin-in-ker})
and $c$ commutes with $\beta$, we have 
\[
\theta([1+x,c(1+y)])=\psi_{\beta}([1+x,c(1+y)])=\psi(\varpi^{-r}\tr(c\beta xc^{-1}-\beta x))=1,
\]
that is, $\theta$ is stabilised by the element $c(1+y)$, hence by
all of $J^{1}$. We saw at the end of Section~\ref{sec:Chars} that
$J^{1}/H^{1}$ is an elementary abelian $p$-group. Define the alternating
bilinear form 
\[
h_{\beta}:J^{1}/H^{1}\times J^{1}/H^{1}\longrightarrow\mathbb{C}^{\times},\qquad h_{\beta}(xH^{1},yH^{1})=\theta([x,y])=\psi_{\beta}([x,y]).
\]
Let $\overline{R}_{\beta}$ be the radical of the form $h_{\beta}$,
and let $R_{\beta}$ denote the preimage of $\overline{R}_{\beta}$
under the map $J^{1}\rightarrow J^{1}/H^{1}$. If we need to specify
which parabolic subalgebra $\mfA_{*}$ we are working with, we will
write $\overline{R}_{\beta,*}$, $R_{\beta,*}$, for $*\in\{\mathrm{m},\mathrm{M}\}$.
For our purposes, we need to determine the dimension of the unique
representation $\eta\in\Irr(J^{1}\mid\psi_{\beta})$, which by Corollary~\ref{cor:Bushnell-Frohlich}
equals $[J^{1}:R_{\beta}]^{1/2}=[J^{1}/H^{1}:\overline{R}_{\beta}]^{1/2}$.

In order to determine the radical of the form $h_{\theta}$, we need
the following result:
\begin{lem}
\label{lem:theta-commutator-1}Let $x,y\in J^{1}$ and write $x=z_{1}(1+s)$
and $y=z_{2}(1+t)$, where $z_{1},z_{2}\in C\cap U^{1}$ and $s,t\in\mathfrak{P}^{el'}$.
Then
\[
\theta([x,y])=\psi_{\beta}(1+(st-ts)).
\]
\end{lem}
\begin{proof}
Note that for any $s_{1},s_{2}\in\mathfrak{P}^{el'}$ we have $z_{1}s_{1}s_{2}\in s_{1}s_{2}+\mathfrak{P}^{e(r-1)+1}$
and 
\[
z_{1}s_{1}z_{2}s_{2}\in z_{1}(s_{1}+\mathfrak{P}^{el'+1})s_{2}\subseteq z_{1}s_{1}s_{2}+\mathfrak{P}^{e(r-1)+1}\subseteq s_{1}s_{2}+\mathfrak{P}^{e(r-1)+1}.
\]
Thus 
\begin{align*}
[x,y] & \in z_{1}(1+s)z_{2}(1+t)(1-s+s^{2}+\mathfrak{P}^{e(r-1)+1})z_{1}^{-1}(1-t+t^{2}+\mathfrak{P}^{e(r-1)+1})z_{2}^{-1}\\
 & \subseteq(z_{1}z_{2}+z_{1}sz_{2})(1+t)(1-s+s^{2})(z_{1}^{-1}z_{2}^{-1}-z_{1}^{-1}tz_{2}^{-1}+z_{1}^{-1}t^{2}z_{2}^{-1})+\mathfrak{P}^{e(r-1)+1}\\
 & \subseteq(1-z_{1}z_{2}sz_{1}^{-1}z_{2}^{-1}+z_{1}z_{2}tz_{1}^{-1}z_{2}^{-1}-ts+z_{1}sz_{1}^{-1}+st\\
 & \phantom{{{}\subseteq{}}}-z_{2}tz_{2}^{-1}+st-t^{2}-st+t^{2})+\mathfrak{P}^{e(r-1)+1}\\
 & \subseteq1-z_{1}z_{2}sz_{1}^{-1}z_{2}^{-1}+z_{1}z_{2}tz_{1}^{-1}z_{2}^{-1}+z_{1}sz_{1}^{-1}-z_{2}tz_{2}^{-1}+st-ts+\mathfrak{P}^{e(r-1)+1}.
\end{align*}
Using the facts that $z_{1}$ and $z_{2}$ commute with $\beta$ and
that $U^{e(r-1)+1}\subseteq\Ker\theta$, we get 
\begin{align*}
\theta([x,y]) & =\psi_{\beta}(1+(-z_{1}z_{2}sz_{1}^{-1}z_{2}^{-1}+z_{1}z_{2}tz_{1}^{-1}z_{2}^{-1}+z_{1}sz_{1}^{-1}-z_{2}tz_{2}^{-1}+st-ts))\\
 & =\psi_{\beta}(1+(st-ts)).
\end{align*}
\end{proof}
Note that the above lemma implies that the value of the form $h_{\beta}$
on the elements $x=z_{1}(1+s)$ and $y=z_{2}(1+t)$ does not depend
on $z_{1},z_{2}\in C\cap U^{1}$. Consider the map
\[
\rho:U^{el'}\longrightarrow U^{el'}/U^{el'+1}\longiso\mfA/\mfP,
\]
where the isomorphism is given by $(1+\varpi^{l'}x)U^{e_{\mathrm{}}l'+1}\mapsto x+\mfP$.
Let $\beta+\mfP$ be the image of $\beta$ in $\mfA/\mfP$ under this
map. 
\begin{lem}
\label{lem:Radical}We have 
\[
R_{\beta}=(C\cap U^{1})\cdot\rho^{-1}(C_{\mfA/\mfP}(\beta+\mfP)).
\]
\end{lem}
\begin{proof}
By definition, $x\in R_{\beta}$ if and only if $\theta([x,y])=1$
for all $y\in J^{1}$. Writing $x=z_{1}(1+s)$, $y=z_{2}(1+t)$ as
in Lemma~\ref{lem:theta-commutator-1}, we have
\[
\theta([x,y])=\psi_{\beta}(1+(st-ts))=\psi(\varpi^{-r}\tr t(\beta s-s\beta)),
\]
so $\theta([x,y])=1$ for all $y\in J^{1}$ is equivalent to $\psi(\varpi^{-r}\tr(\mfP^{el'}(\beta s-s\beta)))=1$,
that is, $\beta s-s\beta\in(\mfP^{el'})^{\perp}=\mfP^{el'+1}$ (see
Lemma~\ref{lem:characters}). The latter is equivalent to $\rho(1+s)\in C_{\mfA/\mfP}(\beta+\mfP)$
because if we write $s=\varpi^{l'}s_{0}$, we have $\rho(1+s)=s_{0}+\mfP$
and $\beta s-s\beta\in\mfP^{el'+1}$ is then equivalent to $\beta s_{0}-s_{0}\beta\in\mfP$,
that is, $s_{0}+\mfP\in C_{\mfA/\mfP}(\beta+\mfP)$. Thus we have
shown that $x=z_{1}(1+s)\in R_{\beta}$ if and only if $1+s\in\rho^{-1}(C_{\mfA/\mfP}(\beta+\mfP))$.
\end{proof}
\begin{lem}
\label{lem: dim-etam-etaM} With notation as above, the following
holds:
\begin{enumerate}
\item \label{enu:index}$[J^{1}:R_{\beta}]=\left|\frac{\mfA/\mfP}{C_{\mfA/\mfP}(\beta+\mfP)}\right|$.
\item \label{enu:dim-Max}Suppose that $\mfA=\Amax$. Then the form $h_{\beta}$
on $\Jmax^{1}/\Hmax^{1}$ is non-degenerate. Thus, for every $\thetaM$,
there exists a unique $\etaM\in\Irr(J^{1}\mid\thetaM)$ and 
\[
\dim\etaM=q^{N(N-1)/2}.
\]
\item \label{enu:dim-Min}Suppose that $\mfA=\Amin$. Then, for every extension
$\tilde{\theta}_{\mathrm{m}}$ of $\thetam$ to $R_{\beta}$, there
exists a unique $\etam\in\Irr(J^{1}\mid\tilde{\theta}_{\mathrm{m}})$
and
\[
\dim\etam=\prod_{i=1}^{h}q^{d_{i}m_{i}(d_{i}-1)/2}.
\]
 
\end{enumerate}
\end{lem}
\begin{proof}
By Lemma~\ref{lem:Radical}, we have 
\begin{align*}
J^{1}/R_{\beta} & \cong\frac{U^{el'}}{(C\cap U^{el'})\rho^{-1}(C_{\mfA/\mfP}(\beta+\mfP))}=\frac{U^{el'}}{\rho^{-1}(C_{\mfA/\mfP}(\beta+\mfP))}\\
 & \cong\frac{U^{el'}/U^{el'+1}}{\rho^{-1}(C_{\mfA/\mfP}(\beta+\mfP))/U^{el'+1}}\cong\frac{\mfA/\mfP}{C_{\mfA/\mfP}(\beta+\mfP)}.
\end{align*}
Next, suppose that $\mfA=\Amax=\mfg_{r}$. We then have $\mfA/\mfP=\mfg_{1}$
and $\beta+\mfP=\bar{\beta}$. By Lemma~\ref{lem:Radical}, we need
to show that $\rho^{-1}(C_{\mfg_{1}}(\bar{\beta}))\subseteq H^{1}$,
and this holds if and only if the map 
\begin{equation}
C_{K^{l'}}(\beta)\longrightarrow C_{\mfg_{1}}(\bar{\beta})\label{eq:surject-centr}
\end{equation}
induced by $\rho$, is surjective. To show the latter, first note
that the map $C_{K^{l'}}(\beta)\rightarrow C_{\mfg_{l}}(\beta_{l})$,
$1+\pi^{l'}x\mapsto x_{l}$ is easily seen to be an isomorphism. Now,
$C_{\mfg_{l}}(\beta_{l})$ is an $\mfo_{l}$-module so the map 
\[
C_{\mfg_{l}}(\beta_{l})\longrightarrow C_{\mfg_{1}}(\bar{\beta})\cong C_{\mfg_{l}}(\beta_{l})/\mfp C_{\mfg_{l}}(\beta_{l})
\]
given by $x\mapsto\bar{x}=x+\mfP$ is surjective. Hence the map \eqref{eq:surject-centr}
is surjective, so the form $h_{\beta}$ is indeed non-degenerate on
$\Jmax^{1}/\Hmax^{1}$. By Lemma~\ref{lem:Bushnell-Frohlich}, there
exists a unique $\etaM\in\Irr(J^{1}\mid\thetaM)$ and by \ref{enu:index}
together with Lemma~\ref{lem:Centr_A/P-orders}, its dimension is
\[
\dim\etaM=[\Jmax^{1}:R_{\beta,\mathrm{M}}]^{1/2}=\left|\mfg_{1}/C_{\mfg_{1}}(\bar{\beta})\right|^{1/2}=q^{N(N-1)/2}.
\]
Finally, suppose that $\mfA=\Amin$, and let $\tilde{\theta}_{\mathrm{m}}$
be an arbitrary extension of $\thetam$ to $R_{\beta}$. By Corollary~\ref{cor:Bushnell-Frohlich},
there exists a unique $\etam\in\Irr(J^{1}\mid\tilde{\theta}_{\mathrm{m}})$
and by \ref{enu:index} together with Lemma~\ref{lem:Centr_A/P-orders},
its dimension is 
\begin{multline*}
\dim\etam=[\Jmin^{1}:R_{\beta,\mathrm{m}}]^{1/2}=\left|\frac{\Amin/\Pmin}{C_{\Amin/\Pmin}(\betam)}\right|^{1/2}\\
=\left(\frac{\prod_{i=1}^{h}q^{d_{i}^{2}m_{i}}}{q^{N}}\right)^{1/2}=\prod_{i=1}^{h}q^{(d_{i}^{2}m_{i})/2}\prod_{i=1}^{h}q^{-d_{i}m_{i}/2}\\
=\prod_{i=1}^{h}q^{d_{i}m_{i}(d_{i}-1)/2}.
\end{multline*}
\end{proof}
\begin{rem}
In the proof of the second part of the above lemma we used the fact
that the map $C_{\Umax^{\emax l'}}(\beta)\longrightarrow C_{\Amax/\Pmax}(\bar{\beta})$
induced by $\rho$ is surjective. We remark that the corresponding
map $C_{\Umin^{\emin l'}}(\beta)\longrightarrow C_{\Amin/\Pmin}(\betam)$
is \emph{not} surjective in general. Consequently, $\eta_{\mathrm{m}}$
is not the only representation containing $\theta_{\mathrm{m}}$.
This will not matter for us, since we will only need the existence
of an $\etam$. However, the uniqueness of $\etaM$ expressed in the
above lemma is crucial, because without it we would not be able to
deduce that the representation $\eta$ defined below is an extension
of $\etaM$.

Note also that a quicker proof of the non-degeneracy of $h_{\beta}$
in the second part of the lemma is to observe that $[\Jmax^{1}:R_{\beta,\mathrm{M}}]=[\Jmax^{1}:\Hmax^{1}]$.
In the above proof, we have emphasised the surjectivity of \eqref{eq:surject-centr}.
\end{rem}
We now prove a series of lemmas whose purpose is to show the existence
of an extension of $\etaM$ to $CK^{l'}$. Once this is achieved,
the construction is easily completed.
\begin{lem}
\label{lem:eta-extn-etaM} Let 
\[
\eta:=\Ind_{\Jmin^{1}}^{\JmM}\eta_{\mathrm{m}}.
\]
Then $\dim\eta=\dim\etam$ and thus $\eta$ is an extension of $\etaM$.
\end{lem}
\begin{proof}
We first need to determine the dimension of the induced representation
Given Lemma~\ref{lem: dim-etam-etaM}, we only need to compute the
index $[\JmM:\Jmin^{1}]$. We claim that $C\cap K^{l'}\subseteq\Umin^{\emin l'}$.
Indeed, we have the relations
\[
C\cap K^{l'}=1+\mfp^{l'}\rho_{r,l}^{-1}(C_{\mfg_{l}}(\beta_{l}))\subseteq1+\mfp^{l'}\Amin,
\]
where the inclusion follows from our assumption that $\beta\in\Amin$
together with Lemma~\eqref{lem:Centr-in-Am}, guaranteeing that $\rho_{l,1}(C_{\mfg_{l}}(\beta_{l}))\subseteq\overline{\mfA}_{\mathrm{m}}$,
and so $\rho_{r,l}^{-1}(C_{\mfg_{l}}(\beta_{l}))\subseteq\rho_{r,1}^{-1}(\overline{\mfA}_{\mathrm{m}})=\Amin$.
We now have

\begin{align*}
\JmM/\Jmin^{1}\cong\frac{K^{l'}}{(C\cap K^{l'})\Umin^{\emin l'}} & =\frac{K^{l'}}{\Umin^{\emin l'}},
\end{align*}
where the equality follows from the above claim. Furthermore, the
map 
\[
\frac{K^{l'}}{\Umin^{\emin l'}}\longrightarrow\mfg_{1}/\overline{\mfA}_{\mathrm{m}},\qquad(1+\pi^{l'}x)\Umin^{\emin l'}\longmapsto\bar{x}+\overline{\mfA}_{\mathrm{m}}
\]
(recall that $\overline{\mfA}_{\mathrm{m}}$ is the image of $\Amin$
in $\mfg_{1}$) is an isomorphism, and we have 
\[
\left|\mfg_{1}/\overline{\mfA}_{\mathrm{m}}\right|=\frac{q^{N^{2}}}{q^{N^{2}/2}\prod_{i=1}^{h}q^{d_{i}^{2}m_{i}/2}}=q^{N^{2}/2}\prod_{i=1}^{h}q^{-d_{i}^{2}m_{i}/2}.
\]
Thus, by Lemma~\ref{lem: dim-etam-etaM}, we have 
\begin{align*}
\dim\eta & =\dim\etam\cdot\left|\mfg_{1}/\overline{\mfA}_{\mathrm{m}}\right|=\prod_{i=1}^{h}q^{d_{i}m_{i}(d_{i}-1)/2}q^{N^{2}/2}\prod_{i=1}^{h}q^{-d_{i}^{2}m_{i}/2}\\
 & =q^{N(N-1)/2}=\dim\etaM.
\end{align*}
By construction, the representation $\eta$ contains $\thetaM$ on
restriction to $\Hmax^{1}$. Hence, the representation $\eta|_{\Jmax^{1}}$
contains $\thetaM$ on restriction to $\Hmax^{1}$. By Lemma~\ref{lem: dim-etam-etaM}\,\ref{enu:dim-Max}
$\etaM$ is the unique representation of $\Jmax^{1}$ which contains
$\thetaM$, so it follows that $\eta$ contains $\etaM$ on restriction
to $\Jmax^{1}$. The equality of dimensions $\dim\eta=\dim\etaM$
then forces $\eta|_{\Jmax^{1}}=\etaM$, so that $\eta$ is an extension
of $\etaM$.
\end{proof}
\begin{lem}
\label{lem:p-Sylow-extension}Let $G$ be a finite group, $N$ a normal
$p$-subgroup of $G$, and $P$ a Sylow $p$-subgroup of $G$. Suppose
that $\chi\in\Irr(N)$ is stabilised by $G$ and that $\chi$ has
an extension to $P$. Then $\chi$ has an extension to $G$.
\end{lem}
\begin{proof}
By \cite[(11.31)]{Isaacs}, $\chi$ will extend to $G$ if it extends
to every $H\subseteq G$ such that $H/N$ is a Sylow subgroup of $G/N$.
By assumption, $\chi$ has an extension, say $\tilde{\chi}$, to $P$,
so if $Q$ is any other Sylow $p$-subgroup of $G$, then $Q=gPg^{-1}$
for some $g\in G$, and so $^{g}\tilde{\chi}$ is an extension of
$\chi$ to $Q$, because $^{g}\chi=\chi$. Suppose now that $P'\subseteq G$
is a subgroup such that $P'/N$ is a Sylow $p'$-subgroup of $G/N$,
for some prime $p'\neq p$. Then $p$ does not divide the index $[P':N]$,
so by a theorem of Gallagher (see \cite[Theorem~6]{Gallagher} or
\cite[(8.16)]{Isaacs}) $\chi$ extends to $P'$. Thus $\chi$ extends
to $G$.
\end{proof}
\begin{lem}
\label{lem:Stab-etaM}We have $CK^{l'}=\Stab_{CK^{l'}}(\etaM)$.
\end{lem}
\begin{proof}
Recall that we saw in Section~\eqref{sec:Parahoric} that $\Hmax^{1}$
and $\Jmax^{1}$ are normal in $CK^{l'}$. Let $g\in CK^{l'}$. Since
$\etaM$ is the unique representation in $\Irr(\Jmax^{1}\mid\thetaM)$,
we have $g\in\Stab_{CK^{l'}}(\etaM)$ if and only if $g\in\Stab_{CK^{l'}}(\thetaM)$,
so we need to show that $g$ stabilises $\thetaM$. To this end, write
$g=zu$ with $z\in C$, $u\in K^{l'}$, and $x=z'v$ with $z'\in C\cap K^{1}$,
$v\in K^{l}$. Then 
\[
gxg^{-1}=z'\cdot z([z'^{-1},u](uvu^{-1}))z^{-1},
\]
where $z([z'^{-1},u](uvu^{-1}))z^{-1}\in K^{l}$. Write $\thetaM=\theta_{0}\psi_{\beta}$
as in Remark~\ref{rem:theta_0psi_beta}. Then
\begin{align*}
\thetaM(gxg^{-1}) & =\theta_{0}(z')\psi_{\beta}(z([z'^{-1},u](uvu^{-1}))z^{-1})\\
 & =\theta_{0}(z')\psi_{\beta}([z'^{-1},u])\psi_{\beta}(v)\\
 & =\thetaM(x)\psi_{\beta}([z'^{-1},u]),
\end{align*}
where the second equality follows since $CK^{l'}$ stabilizes $\psi_{\beta}$.
To show that $\thetaM$ is stabilized by $g$ it thus remains to show
that $\psi_{\beta}([z'^{-1},u])=1$. To this end, write $u=1+s$,
with $s\in\mfp^{l'}\mfg_{r}$ and observe that 
\[
[z'^{-1},u]=z'^{-1}(1+s)z'(1-s+s^{2})=z'^{-1}(1+s)z'-s,
\]
where we have used that $z'^{-1}sz's=s^{2}$. Thus $\psi_{\beta}([z'^{-1},u])=\psi(\varpi^{-r}\tr\beta(z'^{-1}sz'-s))=\psi(\varpi^{-r}\tr(z'^{-1}(\beta s)z'-\beta s))=1$,
as required.
\end{proof}
\begin{thm}
\label{thm:Main-regular}Suppose that the orbit $\Omega_{l'}$ consists
of regular elements and let $\beta\in\Omega_{r}\cap\Amin$. Then,
for any extension $\thetaM$ of $\psi_{\beta}$, the representation
$\etaM$ has an extension $\hatetaM$ to $CK^{l'}$, where $C=C_{G_{r}}(\beta)$.
Any representation in $\Irr(G_{r}\mid\psi_{\beta})$ is of the form
\[
\pi(\thetaM,\hatetaM):=\Ind_{CK^{l'}}^{G_{r}}\hatetaM,
\]
for some $\thetaM$ and $\hatetaM$, and if another representation
$\pi(\thetaM',\hatetaM')\in\Irr(G_{r}\mid\psi_{\beta})$ is isomorphic
to $\pi(\thetaM,\hatetaM)$, then $\thetaM\cong\thetaM'$ and $\hatetaM\cong\hatetaM'$. 
\end{thm}
\begin{proof}
The first assertion follows from Lemma~\ref{lem:p-Sylow-extension},
using Lemma~\ref{lem:p-Sylow}, Lemma~\ref{lem:Stab-etaM} and Lemma~\ref{lem:eta-extn-etaM}.

Choose $\beta\in\Omega_{r}$. Any representation in $\Irr(\Hmax^{1}\mid\psi_{\beta})$
is of the form $\thetaM$, any representation in $\Irr(\Jmax^{1}\mid\thetaM)$
is isomorphic to $\etaM$ and any representation in $\Irr(CK^{l'}\mid\etaM)$
is of the form $\hatetaM$. Thus any representation in $\Irr(CK^{l'}\mid\psi_{\beta})$
is of the form $\hatetaM$. By a standard result from Clifford theory
of finite groups \cite[(6.11)]{Isaacs}, this means that any representation
in $\Irr(G_{r}\mid\psi_{\beta})$ is of the form $\pi(\thetaM,\hatetaM)$. 

Suppose that $\pi(\thetaM,\hatetaM)$ and $\pi(\thetaM',\hatetaM')$
are isomorphic. By \cite[(6.11)]{Isaacs} we must have $\hatetaM\cong\hatetaM'$.
Thus the corresponding representations $\etaM$ and $\etaM'$ are
isomorphic, and since these uniquely determine $\thetaM$ and $\thetaM'$,
respectively, we must have $\thetaM\cong\thetaM'$.
\end{proof}
Note that even though $\eta$ is an extension of $\etaM$ to $\JmM$
and $\hatetaM$ is an extension of $\etaM$ to $CK^{l'}$, we do not
know, and do not need to know, whether $\hatetaM$ is an extension
of $\eta$. 

\section{\label{sec:Concluding-remarks}Concluding remarks}

The obstruction to extending a representation in $\Irr(K^{l'}\mid\psi_{\beta})$
to $CK^{l'}$ is given by an element in the Schur multiplier $H^{2}(\F_{q}[\bar{\beta}]^{\times},\mathbb{C}^{\times})$.
Takase conjectured that for $p=\chara\F_{q}$ large enough, this element
is trivial; see \cite[Conjecture~4.6.5]{Takase-16}. Using our main
result, we deduce a strong form of Takase's conjecture, valid for
any prime $p$ (this is also proved, for $p\neq2$, in \cite{KOS}). 
\begin{cor}
Suppose $\beta\in\mfg_{r}$ is regular. Every representation in $\Irr(K^{l'}\mid\psi_{\beta})$
extends to $CK^{l'}$, and hence Takase's conjecture \cite[Conjecture~4.6.5]{Takase-16}
holds for $\F_{q}$ of arbitrary characteristic.
\end{cor}
\begin{proof}
Let $\sigma\in\Irr(K^{l'}\mid\psi_{\beta})$. It is straightforward
(cf.~\cite[Proposition~4.2]{Hill_regular}) that the bilinear form
$h_{\beta}$ on $K^{l'}/K^{l}$ defined by $\psi_{\beta}$ has radical
$(C_{G_{r}}(\beta)\cap K^{l'})K^{l}/K^{l}$. We have 
\[
|(C_{G_{r}}(\beta)\cap K^{l'})K^{l}/K^{l}|=|C_{\mfg_{1}}(\bar{\beta})|=q^{N},
\]
so that by Corollary~\ref{cor:Bushnell-Frohlich}, $\sigma$ has
dimension $q^{N(N-1)/2}$. Let $\tilde{\sigma}\in\Irr(CK^{l'}\mid\sigma)$
be a constituent of $\Ind_{K^{l'}}^{CK^{l'}}\sigma$. By Theorem~\ref{thm:Main-regular},
any representation in $\Irr(CK^{l'}\mid\psi_{\beta})$, so in particular
$\tilde{\sigma}$, is an extension of some representation in $\Irr(\Jmax^{1}\mid\psi_{\beta})$.
Let $\etaM\in\Irr(\Jmax^{1}\mid\psi_{\beta})$ be such that $\tilde{\sigma}$
is an extension of $\etaM$. By Lemma~\ref{lem: dim-etam-etaM} $\dim\etaM=q^{N(N-1)/2}$.
Thus, $\tilde{\sigma}$ is an irreducible representation of dimension
$q^{N(N-1)/2}$ whose restriction to $K^{l'}$ has a constituent $\sigma$
of the same dimension. Thus $\tilde{\sigma}$ is an extension of $\sigma$.
\end{proof}
In \cite{Alex_smooth_reps_GL2} the construction of split regular
representations of $\GL_{2}(\mfo_{r})$ appealed to Hill's construction
\cite[Theorem~4.6]{Hill_regular}. Since Takase \cite{Takase-16}
has realised that this construction does not produce all the split
regular representations, one should view the construction of the current
paper as superseding that of \cite{Alex_smooth_reps_GL2}, while at
the same time unifying the split regular case with the cuspidal.

In \cite{KOS} the dimensions and multiplicities of the regular representations
of $\GL_{N}(\mfo_{r})$ were determined for $p\neq2$. This can also
be done using Theorem~\ref{thm:Main-regular}, and our construction
implies that the dimension and multiplicity formulas there remain
valid for $p=2$.

\bibliographystyle{alex}
\bibliography{alex}

\end{document}